\documentclass{amsart}
\theoremstyle{plain}
\newtheorem{thm}{Theorem}

 \newtheorem{prop}{Proposition}
\newtheorem{lem}{Lemma}
\newtheorem{cor}{Corollary}

\errorcontextlines=0 \numberwithin{equation}{section}

\begin{document}
\large
\title[Heat flow method to Lichnerowicz type equation  ]
{Heat flow method to Lichnerowicz type equation on closed manifolds}
\author{Li Ma, Yuhua Sun}

\address{Ma: Department of Mathematical Sciences \\
Tsinghua University \\
Beijing 100084 \\
China}

\email{lma@math.tsinghua.edu.cn}

\dedicatory{}
\date{}
\thanks{The research is partially supported by the National Natural Science
Foundation of China 10631020 and SRFDP 20090002110019 }

\begin{abstract}
In this paper, we establish existence results for positive solutions
to the Lichnerowicz equation of the following type in closed
manifolds
\begin{align*}
-\Delta u=A(x)u^{-p}-B(x)u^{q},\quad in\quad M,
\end{align*}
where $p>1, q>0$, and $A(x)>0$, $B(x)\geq0$ are given smooth
functions. Our analysis is based on the global existence of positive
solutions to the following heat equation
\begin{align*}
\left\{\begin{array}{ll}
u_t-\Delta u=A(x)u^{-p}-B(x)u^{q},\quad in\quad M\times\mathbb{R}^{+},\\
u(x,0)=u_0,\quad in\quad M
\end{array}
\right.
\end{align*}
with the positive smooth initial data $u_0$.


\end{abstract}

\maketitle

\section{Introduction}

The aim of this paper is to give some existence results for positive
solutions to the Lichnerowicz type equations in a compact Riemannian
manifold $(M,g)$. In the interesting paper \cite{Hebey08}, the
authors have proved an existence result for Lichnerowicz equation on
closed manifolds by the mini-max method. In \cite{MX}, Li Ma and
Xingwang Xu have investigated the
 negative-positive mixed
index case in elliptic equations, as $\Delta u+f(u)=0$, where
$f(u)=u^{-p-1}-u^{p-1}$, they gained a uniform bound and a
non-existence result for positive solutions to the Lichnerowicz
equation in $\mathbb{R}^n$. Some further interesting results on the
Lichnerowicz equation on a Riemannian manifold $(M,g)$ of dimension
$n\geq 3$ have been obtained in \cite{CIP06}, \cite{CIP}, and
\cite{Hebey08}. Motivated by these papers, we mainly consider heat
flow methods for the Lichnerowcz type equations in the
negative-positive mixed index cases.

 Here we recall the Lichnerowicz equation
on manifolds. Given a smooth symmetric 2-tensor $\sigma$, a smooth
vector field $W$, and a triple data $(\pi, \tau,\varphi)$ of smooth
functions on $M$. Set
$$
c_n=\frac{n-2}{4(n-1)}, \; \; p=\frac{2n}{n-2},
$$
and let
$$
R_{\gamma,\varphi}=c_n(R(\gamma)-|\nabla \varphi|^2_{\gamma}), \; \;
A_{\gamma,W,\pi}=c_n(|\sigma+DW|_{\gamma}^2+\pi^2)
$$
and
$$
B_{\tau,\varphi}=c_n(\frac{n-1}{n}\tau^2-V(\varphi))
$$
where $V:\mathbf{R}\to\mathbf{R}$ is a given smooth function and
$R(\gamma)$ is the scalar curvature function of $\gamma$. Then the
Lichnerowicz equation for the Einstein-scalar conformal data
$(\gamma,\sigma,\pi,\tau,\varphi)$ with the given vector field $W$
is
\begin{equation}\label{lich0}
\Delta_{\gamma}
u-R_{\gamma,\varphi}u+A_{\gamma,W,\pi}u^{-p-1}-B_{\tau,\varphi}u^{p-1}=0,
\;  u>0,
\end{equation}
where $\Delta_{\gamma}$ is the Laplacian operator of $\gamma$.
 As in
\cite{MX}, we use the convention that $\Delta_{\gamma} u=u''$ on the
real line $\mathbf{R}$. Note that $A_{\gamma,W,\pi}\geq 0$.

 In this paper, we mainly deal with the heat equation of the following type
\begin{align}\label{eq3}
\left\{\begin{array}{ll} u_t-\Delta u
=A(x)u^{-p}-B(x)u^{q},\quad in\quad M\times \mathbb{R}^{+},\\
u(x,0)=u_{0},\quad in\quad M,
\end{array}
\right.
\end{align}
where $p>1,q>0$; $A(x)>0$ and $B(x)$  are both smooth functions in
M.

Applying sub-sup solution method on manifolds, we can prove
\begin{thm}\label{thm2}
If $B(x)>0$, and the initial data $u_0$ is a smooth function with
$u_0>0$, then there exists a unique smooth solution to the problem
(\ref{eq3}). As $t\to \infty$, we can find a subsequence $(t_j)$
such that
$$
u(t_j)\to u_{\infty}, \ \ in \ \ H^1(M)
$$
where $u_{\infty}$ is the positive solution to
$$
-\Delta u =A(x)u^{-p}-B(x)u^{q},\quad in\quad M
$$
\end{thm}

If $B(x)\geq0$, and $B(x)$ does have zero points in $M$, then
 we have the following result
\begin{thm}\label{thm3}
Assume $p>1, 1<q\leq \frac{n+2}{n-2}$. If $B(x)\geq0$, $\int_M
B(x)^{-\frac{1}{q}}dx<+\infty$, and the initial data $u_0$ is a
positive smooth function in $M$, then there exists a unique smooth
solution $u(x,t)$ to problem (\ref{eq3}). In addition, one can take
a sequence $t_i\to\infty$ such that $u(x,t_i)\rightharpoonup
u_{\infty}$ in $H^1(M)$, and $u_{\infty}\in C^2$ solves the
Lichnerowicz type equation
\begin{align}\label{lich2}
-\Delta u_{\infty}=A(x)u_{\infty}^{-p}-B(x)u_{\infty}^{q},\quad
in\quad M.
\end{align}
\end{thm}

We gain the existence results by the heat flow method. We point
out that in E.Hebey's recent work (see Appendix), he has explored
the Lichnerowicz equation using sub-sup solution method, but this
method can't go through in Theorem \ref{thm3}. For some special
cases we can give more detailed analysis of the heat equations.

This paper is organized as follows. In section \ref{sect2} we prove
Theorems \ref{thm2} and \ref{thm3}. In section \ref{sect3}, some
precise results are presented in special cases.

\section{Heat flows methods to the Lichnerowicz type equation and proofs
}\label{sect2}

This section is mainly devoted to the proofs of Theorems \ref{thm2}
and \ref{thm3}. To complete the proofs, we should make some
preparations by establishing sub-sup solution method on manifolds.
We take the same argument as in D.H.Sattinger's paper \cite{SAT},
but the difference is that D.H.Sattinger has dealt with sub-sup
solutions only in Euclidean spaces, and for convenience of readers,
we present his some results on manifolds. We point out that the
techniques in the proofs are almost the same as in D.H.Sattinger's.

Consider the initial problem to the following heat equation in
closed manifold
\begin{align}\label{sub-sup equation}
\left\{\begin{array}{ll} u_t-\Delta u=f(x,u),\quad in\quad M\times
\mathbb{R}^{+},\\
u(x,0)=u_0,\quad in\quad M.
\end{array}
\right.
\end{align}
where $f(x,u)=A(x)u^{-p}-B(x)u^{q}$, and $A(x),B(x), u_0>0$ are
smooth functions.

We call the function $\varphi_0(x,t)>0$ a sup-solution to
(\ref{sub-sup equation}) if
\begin{align*}
\left\{\begin{array}{ll} \partial_t\varphi_0-\Delta
\varphi_0-f(x,\varphi_0)\geq0,\quad in\quad M\times
\mathbb{R}^{+},\\
\varphi_0(x,0)\geq u_0,\quad in\quad M.
\end{array}
\right.
\end{align*}
Similarly the sub-solution $\psi_0(x,t)>0$ is defined by if
\begin{align}\label{sub1}
\left\{\begin{array}{ll} \partial_t\psi_0-\Delta
\psi_0-f(x,\psi_0)\leq0,\quad in\quad M\times
\mathbb{R}^{+},\\
\psi_0(x,0)\leq u_0,\quad in\quad M.
\end{array}
\right.
\end{align}

Next, we define the mapping $\varphi_1=\Lambda \varphi_0$ by
\begin{align*}
\left\{\begin{array}{ll}
\partial_t\varphi_1-\Delta\varphi_1+\Omega\varphi_1
=f(x,\varphi_0)+\Omega\varphi_0,\quad in\quad M\times
\mathbb{R}^{+},\\
\varphi_1(x,0)=u_0,\quad in\quad M.
\end{array}
\right.
\end{align*}
where $\Omega$ is constant chosen large enough.

\begin{lem}\label{par-thm1}
Given a positive sup-solution $\varphi_0(x,t)$ and a positive
sub-solution $\psi_0(x,t)$ to problem (\ref{sub-sup equation}) in
closed manifold M. Define sequences $\{\varphi_n\}$ and $\{\psi_n\}$
inductively by $\varphi_{n+1}=\Lambda\varphi_n$,
$\psi_{n+1}=\Lambda\psi_n$. If $\Omega$ is large enough so that
\begin{align*}
\frac{\partial f}{\partial u}(x,u)+\Omega>0\quad on\quad
\min_{M\times(0,T)}\psi_0<u<\max_{M\times(0,T)}\varphi_0,
\end{align*}
then the sequences $\{\varphi_n\}$ and $\{\psi_n\}$ are monotone
increasing and decreasing respectively. As $n\to\infty$, they both
tend to a unique fixed point $u=\Lambda u$, which is a smooth
solution of
\begin{align*}
\left\{\begin{array}{ll}
\partial_t u-\Delta u=f(x,u),\quad in\quad M\times(0,T)\\
u(x,0)=u_0,\quad in\quad M.
\end{array}
\right.
\end{align*}
\end{lem}

\begin{proof}
We can take $\Omega>0$ large enough, so that
\begin{align}\label{monotone1}
\frac{\partial f}{\partial u}(x,u)+\Omega>0
\end{align}
for all $x\in M$, and $\min \psi_0\leq u\leq \max \varphi_0$.

Recall that the sequence $\{\psi_k\}$ is defined as follows
\begin{align}\label{iteration1}
\left\{\begin{array}{ll} \partial_t\psi_{k+1}-\Delta
\psi_{k+1}+\Omega\psi_{k+1}=f(x,\psi_k)+\Omega\psi_k,\quad in\quad
M\times
\mathbb{R}^{+},\\
\psi_{k+1}(x,0)=u_0,\quad in\quad M.
\end{array}
\right.
\end{align}
and the sequence $\{\varphi_k\}$ is similarly defined.

We claim
\begin{align}\label{sequence}
\psi_0\leq\psi_1\cdots\leq\psi_k\leq\cdots\leq\varphi_k\cdots\leq\varphi_1\leq\varphi_0.
\end{align}

To confirm this, first note for $k=0$ that
\begin{align}\label{sub2}
\left\{\begin{array}{ll} \partial_t\psi_1-\Delta
\psi_1+\Omega\psi_1=f(x,\psi_0)+\Omega\psi_0,\quad in\quad M\times
\mathbb{R}^{+},\\
\psi_1(x,0)=u_0,\quad in\quad M.
\end{array}
\right.
\end{align}
Subtracting (\ref{sub2}) from (\ref{sub1}), we find
\begin{align*}
\left\{\begin{array}{ll}
\partial_t(\psi_1-\psi_0)-\Delta(\psi_1-\psi_0)+\Omega(\psi_1-\psi_0)\geq0,\quad in\quad M\times
\mathbb{R}^{+},\\
(\psi_1-\psi_0)(x,0)\geq0,\quad in\quad M.
\end{array}
\right.
\end{align*}
 Applying the maximum principle, we know
\begin{align*}
\psi_0\leq\psi_1,\quad in \quad M\times \mathbb{R}^{+}.
\end{align*}

We now assume inductively
\begin{align}\label{monotone2}
\psi_{k-1}\leq\psi_{k},\quad in \quad M\times \mathbb{R}^{+}.
\end{align}
By (\ref{iteration1}) we  have
\begin{align}\label{iteration2}
\left\{\begin{array}{ll}
\partial_t(\psi_{k+1}-\psi_k)-\Delta(\psi_{k+1}-\psi_k)
+\Omega(\psi_{k+1}-\psi_k)\\\quad
=f(x,\psi_{k})-f(x,\psi_{k-1})+\Omega(\psi_{k}-\psi_{k-1}),\quad
in\quad M\times
\mathbb{R}^{+},\\
(\psi_{k+1}-\psi_k)(x,0)=0,\quad in\quad M.
\end{array}
\right.
\end{align}
From (\ref{monotone1}), (\ref{monotone2}), using the maximum
principle, we have
\begin{align*}
\psi_k\leq\psi_{k+1},\quad in\quad M\times \mathbb{R}^{+}.
\end{align*}
Hence similar argument can be applied to $\{\varphi_k\}$, we can
deduce $\varphi_{k+1}\leq\varphi_k$.

In light of this, we claim if $\psi_k\leq\varphi_k$, then
$\psi_{k+1}\leq\varphi_{k+1}$. From (\ref{iteration1}), we have
\begin{align}\label{compare}
\left\{\begin{array}{ll}
\partial_t(\varphi_{k+1}-\psi_{k+1})-\Delta(\varphi_{k+1}-\psi_{k+1})
+\Omega(\varphi_{k+1}-\psi_{k+1})\\
=f(x,\varphi_{k})-f(x,\psi_{k})+\Omega(\varphi_{k}-\psi_{k}),\quad
in\quad M\times
\mathbb{R}^{+},\\
(\varphi_{k+1}-\psi_{k+1})(x,0)=0,\quad in\quad M.
\end{array}
\right.
\end{align}
Since
\begin{align}
f(x,\varphi_{k})-f(x,\psi_{k})+\Omega(\varphi_{k}-\psi_{k})
=(f_u(x,\eta)+\Omega)(\varphi_{k}-\psi_{k}),
\end{align}
where $\eta$ is between $\psi_k$ and $\varphi_k$, applying the
maximum principle, we deduce from (\ref{monotone1}), (\ref{compare})
and the assumption $\psi_k\leq\varphi_k$ that
\begin{align}\label{sequence1}
\psi_{k+1}\leq\varphi_{k+1}.
\end{align}
Then (\ref{sequence}) is proved.

Since $\psi_0(x,t), \varphi_0$ are globally bounded, we define
\begin{align*}
u(x,t):=\lim_{k\to\infty}\psi_k(x,t)\\
v(x,t):=\lim_{k\to\infty}\varphi_k(x,t).
\end{align*}
 By (\ref{sequence1}), obviously we have
\begin{align}\label{sequence2}
u(x,t)\leq v(x,t),\quad in\quad M\times \mathbb{R}^{+}.
\end{align}
Applying the monotone convergence theorem, we have
\begin{align}
\left\{\begin{array}{ll}
\psi_k(x,t)\to u(x,t),\quad in \quad L^2(M),\\
\varphi_k(x,t)\to v(x,t),\quad in \quad L^2(M).
\end{array}
\right.
\end{align}
By (\ref{monotone1}), we can deduce
\begin{align}
||f(x,\psi_k)||_{L^2(M)}\leq C(1+||\psi_k||_{L^2(M)}),
\end{align}
where $C$ depends only on $A(x), B(x)$, and $M$.

 Furthermore letting $k\to\infty$, we have
\begin{align}\label{sub-solution}
\left\{\begin{array}{ll} u_t-\Delta u=f(x,u),\quad in\quad M\times
\mathbb{R}^{+},\\
u(x,0)=u_0,\quad in\quad M,
\end{array}
\right.
\end{align}
and
\begin{align}\label{sup-solution}
\left\{\begin{array}{ll} v_t-\Delta v=f(x,v),\quad in\quad M\times
\mathbb{R}^{+},\\
v(x,0)=v_0,\quad in\quad M.
\end{array}
\right.
\end{align}
Subtract (\ref{sub-solution}) from (\ref{sup-solution}), we have
\begin{align}\label{unique}
\left\{\begin{array}{ll} (u-v)_t-\Delta
(u-v)+\Omega(u-v)=f(x,u)-f(x,v)+\Omega(u-v),\; in\; M\times
\mathbb{R}^{+},\\
(u-v)(x,0)=0,\quad in\quad M.
\end{array}
\right.
\end{align}
Observe that
\begin{align*}
f(x,u)-f(x,v)+\Omega(u-v)=(f_u(x,\xi)+\Omega)(u-v),
\end{align*}
where $\xi$ is between $u$ and $v$.

From (\ref{monotone1}), (\ref{unique}), (\ref{sequence2}) and using
the maximum principle, we deduce
\begin{align*}
u=v,\quad in\quad M\times \mathbb{R}^{+}.
\end{align*}
By classical parabolic theory, we know that $u$ is a positive smooth
solution.
\end{proof}

\textbf{Proof of Theorem \ref{thm2}:} Let
$\varphi_0=S(\max\frac{A(x)}{B(x)})^{\frac{1}{p+q}}$,
$\psi_0=s(\min\frac{A(x)}{B(x)})^{\frac{1}{p+q}}$, where $S\geq1,
0\leq s\leq1$ is determined by the following
\begin{align*}
S(\max\frac{A(x)}{B(x)})^{\frac{1}{p+q}}\geq \max u_0,\\
s(\min\frac{A(x)}{B(x)})^{\frac{1}{p+q}}\leq \min u_0.
\end{align*}
We know that $\psi_0$ and $\varphi_0$ constructed above are the
sub-solution and sup-solution of the problem (\ref{eq3}).  By Lemma
\ref{par-thm1}, we know that problem (\ref{eq3}) has a positive
smooth solution.

The uniqueness can be gained as below. If $u$, $v$ are both smooth
solutions to the problem (\ref{eq3}). Then it follows that
\begin{align*}
\left\{\begin{array}{ll}
(u-v)_t-\Delta (u-v)=A(x)u^{-p}-B(x)u^{q}-A(x)v^{-p}+B(x)v^{q},\;in\;M\times \mathbb{R}^{+}\\
(u-v)(x,0)=0,\quad in\quad M.
\end{array}
\right.
\end{align*}
We can just rewrite it as
\begin{align}
\left\{\begin{array}{ll}
(u-v)_t-\Delta (u-v)+(pA(x)\xi^{-p-1}+qB(x)\eta^{q-1})(u-v)=0,\;in\;M\times \mathbb{R}^{+}\\
(u-v)(x,0)=0,\quad in\quad M,
\end{array}
\right.
\end{align}
where $\xi$, $\eta$ are between $u$ and $v$.

Since $u$, $v$ are positive smooth solutions, using the maximum
principle, we can deduce
\begin{align*}
u=v,\quad in\quad M.
\end{align*}
Because of the uniform bound-ness of $u(x,t)$, the convergence at
$t=\infty$ of $u(x,t)$ is by now standard (see also the proof of
Theorem \ref{thm3} below), so we omit the detail. This completes
the proof of Theorem \ref{thm2}.

\textbf{Proof of Theorem \ref{thm3}:} If the term $B(x)$ in equation
(\ref{eq3}) has zero points, letting $Z=\{x\in M; B(x)=0\}$, the
sub-sup solution method above can not be used. Then we add an
positive epsilon to $B(x)$ such that we can use the method to get an
approximate solutions (\ref{eq3}). Let $B_1=\sup_{M}B(x)>0$ and
$B_2=\inf_{M}A(x)>0$.

First, for $\epsilon>0$, we consider the following equation
\begin{align}\label{app1}
\left\{\begin{array}{ll}
\partial_tu_{\epsilon}-\Delta
u_{\epsilon}=A(x)u_{\epsilon}^{-p}-(B(x)+\epsilon)u_{\epsilon}^{q},\quad in\quad M\times \mathbb{R}^{+}\\
u_{\epsilon}(x,0)=u_0,\quad in\; M.
\end{array}
\right.
\end{align}
By the sub-super solution method we know that there is a positive
smooth solution, which is denoted by $u_{\epsilon}$, to
(\ref{app1}). Sometimes we may write $u=u_{\epsilon}$.

We claim that there is a uniform constant $C:=C(u_0)>0$ such that
$$ u_{\epsilon}(x,t)\geq C, \ \ in \ \ M\times (0,\infty).
$$
In fact, for any fixed $T>0$, let
$$u(x_0,t_0):=u_{\epsilon}(x_0,t_0)=\inf_{M\times (0,T]} u(x,t).
$$
Then we have
$$
0\geq (\partial_t -\Delta)u_{\epsilon}(x_0,t_0)=
A(x_0)u_{\epsilon}(x_0,t_0)^{-p}-(B(x_0)+\epsilon)u_{\epsilon}(x_0,t_0)^{q}.
$$
Then we have
$$
A(x_0)u_{\epsilon}(x_0,t_0)^{-p}\leq
(B(x_0)+\epsilon)u_{\epsilon}(x_0,t_0)^{q}.
$$
Then
$$0<B_2\leq A(x_0)\leq (B(x_0)+\epsilon)u_{\epsilon}(x_0,t_0)^{q+p}
$$
$$
\leq (B_1+1)u_{\epsilon}(x_0,t_0)^{q+p},
$$
which implies that
$$
u_{\epsilon}(x_0,t_0)\geq C>0
$$
for some uniform constant $C=C(u_0)>0$. Then the claim is true and
this fact will be used implicitly.

Multiplying (\ref{app1}) with $\partial_t u_{\epsilon}$ and
integrating in $M\times (0,t)$ we have
\begin{eqnarray*}
\int_0^t\int_M|\partial_t u_{\epsilon}|^2dxdt-\int_0^t\int_M \Delta
u_{\epsilon}\partial_t u_{\epsilon}dxdt&=&\int_0^t\int_M
A(x)u_{\epsilon}^{-p}\partial_t
u_{\epsilon}dxdt\\
&-&\int_0^t\int_M(B(x)+\epsilon)u_{\epsilon}^{q}\partial_t
u_{\epsilon}dxdt.
\end{eqnarray*}

Rearranging, we deduce
\begin{eqnarray}\label{int1}
\int_0^t\int_M|\partial_t
u_{\epsilon}|^2dxdt+\frac{1}{2}\int_M|\nabla
u_{\epsilon}|^2dx+\frac{1}{p-1}\int_M
A(x)u_{\epsilon}^{-p+1}dx\nonumber\\
+\frac{1}{q+1}\int_M(B(x)+\epsilon)u_{\epsilon}^{q+1}dx=\frac{1}{2}\int_M|\nabla
u_{0}|^2dx+\frac{1}{p-1}\int_M A(x)u_{0}^{-p+1}dx\nonumber\\
+\frac{1}{q+1}\int_M(B(x)+\epsilon)u_{0}^{q+1}dx
\end{eqnarray}

By Poincare's inequality, we obtain
\begin{align*}
\int_M|u-\bar{u}|^2dx\leq C(M)\int_M|\nabla u|^2dx.
\end{align*}
Then it follows that
\begin{align}\label{poincare1}
\int_M u^2dx\leq |M|\bar{u}^2+C(M)\int_M|\nabla u|^2dx,
\end{align}
where $|M|$ is the volume of manifold $M$.

Note
\begin{eqnarray}\label{average}
\bar{u}&=&\frac{1}{|M|}\int_Mudx=\frac{1}{|M|}\int_MB(x)^{-\frac{1}{q+1}}B(x)^{\frac{1}{q+1}}udx\nonumber\\
&\leq&\frac{1}{|M|}(\int_M
B(x)^{-\frac{1}{q}}dx)^{\frac{q}{q+1}}(\int_MB(x)u^{q+1}dx)^{\frac{1}{q+1}}.
\end{eqnarray}
Thus, from (\ref{int1}), (\ref{poincare1}), (\ref{average}) and the
assumption $\int_M B(x)^{-\frac{1}{q}}dx\leq C(M)$, we deduce for
every $\epsilon$,
\begin{align}\label{convergence1}
\left\{\begin{array}{l}
\int_M u_{\epsilon}^2dx\leq C,\\
\int_M|\nabla u_{\epsilon}|^2dx\leq C,\\
\int_0^t\int_M |\partial_t u_{\epsilon}|^2dxdt\leq C,
\end{array}
\right.
\end{align}
where $C$ doesn't depend on $t, \epsilon$. It depends only on
$A(x)$, $B(x)$, $u_0$, and $M$.

Then we find a sequence $\epsilon_j$ with $\epsilon_j\to 0$, and the
positive function $u\in H^1(M)$, such that
\begin{align}\label{convergence2}
\left\{\begin{array}{ll}
u_{\epsilon_j}\rightharpoonup u,\quad in\quad H^1(M),\\
u_{\epsilon_j}\to u,\quad in\quad L^2(M),\\
\partial_tu_{\epsilon_j}\rightharpoonup \partial_t u,
\quad in\quad L^2(M\times\mathbb{R}^{+}).
\end{array}
\right.
\end{align}

Clearly $u$ satisfies
\begin{align*}
\partial_t u-\Delta
u=A(x)u^{-p}-B(x)u^q,\quad in\quad M\times \mathbb{R}^{+}
\end{align*}
with the initial data $u_0$ and the estimates (\ref{convergence1}).

Then we may take a subsequence $t_k\to\infty$ such that
\begin{align}
\partial_t u(x,t_k)\to 0,\quad in\quad L^2(M)
\end{align}
and $u(x,t_k)$ converges weakly to $\breve{{u}}$ in $H^1$,
$\breve{{u}}\in H^1(M)$ is a positive smooth solution to
\begin{align}\label{mali}
-\Delta u=A(x)u^{-p}-B(x)u^q,\quad in\quad M\times \mathbb{R}^{+}.
\end{align}
 Using the standard regularity theory we know that
 $\breve{{u}}$ is the positive smooth solution to (\ref{mali}).

This completes the proof of Theorem \ref{thm3}.

\section{Simple cases:global existence and asymptotic behavior}
\label{sect3}

In this section we shall give more precise information in some
simple cases.

Firstly consider the following problem
\begin{align}\label{eq1}
\left\{\begin{array}{ll} u_t-\Delta u=u^{-p}-u^{q},\quad in\quad M\times \mathbb{R}^{+},\\
u(x,0)=u_{0},\quad in\quad M,
\end{array}
\right.
\end{align}
where $p,q>1$ and $u_0>0$ is a given smooth function in $M$.

When the initial data is small, we have
\begin{prop}\label{pro1}
Let $u(x,t)$ be the solution to the problem (\ref{eq1}). If the
initial data $u_0$ is  such that $0<u_0\leq 1$, then $u(x,t)\to 1$
uniformly as $t\to\infty$ .
\end{prop}

For convenience we introduce some notations
$$u_{min}=\min_{M} u(x,t),\quad u_{max}=\max_{M}u(x,t),$$
if $A(x)$, $B(x)$, and $u_0$ are smooth, by the classical parabolic
theory we know the $\min,\max$ functions are well defined.

\begin{proof}
 Consider the evolutional trend of the
function $u$ and we naturally study the $\min_Mu,\max_Mu$. Recall
$u_0$ is a smooth function on M, by the classical parabolic theory,
we know that  $u$ is also smooth and the generalized derivatives of
$\min_Mu,\max_Mu$ with respect to $t$ is well-defined.

Since $0< u_0\leq1$, supposing $u$ first reaches $u(x,t)=1$ at
$(x_0,t_0)$, then from equation (\ref{eq1}), we have
$$u_{{\max}_t}(x_0,t_0)=\Delta u_{\max}(x_0,t_0)\leq 0,$$
then we deduce
\begin{align}\label{max1}
u\leq u_{\max}\leq 1,\quad t\in\mathbb{R}^{+},
\end{align}

Investigating
$$u_{{\min}_t}=\Delta u_{\min}+u^{-p}-u^{q}\geq 0,$$
there exists a constant $c_0$ such that
\begin{align}\label{min1}
u\geq u_{\min}\geq c_0>0,\quad t\in\mathbb{R}^{+}.
\end{align}
where $c_0$ can equal to $\min_M{u_0}$.

From (\ref{max1}) and (\ref{min1}), we know that the solution $u$ is
uniformly bounded, and the global existence result is obtained.

We claim that $u(x,t)\rightarrow 1$ uniformly as
$t\rightarrow\infty$. If $u_{min}\leq 1-\delta$ for some $\delta>0$,
we actually have
$$
u_{{min}_t}=C(\delta)u_{min}>0,
$$
which implies that
$$
u_{min}(t)\geq exp(C(\delta)t)u_{min}(0)\to \infty
$$
as $t\to\infty$. Hence we have $u\geq 1-\delta$ for $t\geq T$ for
some large $T>0$. Therefore we have uniform convergence property
that $u(x,t)\to 1$ as $t\to\infty$ uniformly.
\end{proof}

When the initial data $u_0$ is large, we have
\begin{prop}\label{pro2}
Let $u(x,t)$  be the solution to the problem (\ref{eq1}). If the
initial data $u_0$ is such that $0<u_0\leq L$, where $L>1$, then we
have that $u(x,t)\to 1$ uniformly in $M$ as $t\to\infty$.
\end{prop}

\begin{proof} Firstly consider the functions $\min_Mu,\max_Mu$. If $u_{\min}\leq 1-\delta$ for
some $\delta>0$, we then have
$$
u_{{\min}_t}=C(\delta)u_{\min}>0,
$$
which implies that
$$
u_{\min}(t)\geq exp(C(\delta)t)u_{\min}(0)\to \infty
$$
as $t\to\infty$. Hence, we have $u\geq 1-\delta$ for $t\geq T$ for
some large $T>0$.

Consider the property of $u_{\max}>1$ with
$$u_{{\max}_t}\leq u_{\max}^{-p}-u_{\max}^{q}\leq 0.$$
Then
$$u_{\max}\leq L.$$
By the similar argument as before, we have $u\leq 1+\delta$ for
$t\geq T_1$ for some large $T_1>0$. Hence we have $u(x,t)\to 1$
uniformly as $t\to \infty$.
\end{proof}

We may also investigate the case where $A(x)=C_0B(x)>0$. We have a
similar result to Proposition \ref{pro1}, \ref{pro2}.
\begin{cor}\label{cor1}
Suppose $A(x)/B(x)\equiv C_0$, and the initial data $u_0$ is a
smooth function with $u_0>0$, then the problem (\ref{eq3}) has a
unique smooth solution u(x,t). Furthermore, $u(x,t)\rightarrow
C_0^{\frac{1}{(p+q)}}$ uniformly as $t\rightarrow\infty$.
\end{cor}
\bigskip
\textbf{Proof of Corollary\;\ref{cor1}} This proof is similar to the
proofs of Propositions \ref{pro1} and \ref{pro2}, so we may omit it.

\section{Appendix}\label{sect4}
Consider the following PDE on a compact Riemannian manifold $(M,g)$
$$
-\Delta_gu+h(x)u=A(x)u^{-p}+B(x)u^q.
$$
Here $h(x$,$A(x)$, and $B(x)$ are smooth functions on $M$. Assume
that $A(x)>0$ and $B(x)<0$.

Hebey's result is below: one can get such a solution in a few lines
 with the sub and super solution method (at least if
 one assumes that $-\Delta_g+h$ is a positive operator).
 Assuming $-\Delta_g+h$ is a positive operator (e.g.
 $h > 0$) one lets $v > 0$ be arbitrary and $u > 0$
 be the solution of $-\Delta_gu+hu = v$. Then
 (1)
 $u_0 = \epsilon u$, $0 < \epsilon \ll 1$ is a
 subsolution for the equation, and

 (2) $u_1 = tu$, $t \gg 1$, is a supersolution for
 the equation.
 Noting that $u_0 < u_1$ one can apply the sub and super
 solution method and one gets a solution to the equation.

\textbf{Acknowledgement} The second author also wants to thank Liang
Cheng for helpful discussion.

\end{document}